\documentclass[journal]{IEEEtran}

\usepackage{multirow}
\usepackage{multicol}
\usepackage{subfigure}
\usepackage{caption}
\usepackage[english]{babel}
\usepackage{graphicx}
\usepackage{framed}
\usepackage[normalem]{ulem}

\usepackage{amsmath, amsthm, amssymb}
\usepackage{amsfonts}
 \usepackage{xcolor}
\usepackage{enumerate}
\usepackage[ruled,vlined]{algorithm2e}
\usepackage[utf8]{inputenc}

\newcommand{\<}{\langle}
\renewcommand{\>}{\rangle}

\theoremstyle{definition}
\newtheorem{theorem}{Theorem}
\newtheorem{lemma}{Lemma}
\newtheorem{prop}{Proposition}

\newtheorem{remark}{Remark}

\newcommand{\bR}{\mathbb{R}}

\DeclareMathOperator*{\argmin}{argmin}

\makeatletter
\newcommand*{\rom}[1]{\expandafter\@slowromancap\romannumeral #1@}
\makeatother

\title{\LARGE \bf
Aggressive Local Search for Constrained Optimal Control Problems with Many Local Minima}

\author{Yuhao Ding, Han Feng and Javad Lavaei
\thanks{This work was supported by grants from AFOSR, NSF EPCN, ONR and DARPA.}
\thanks{Y. Ding, H. Feng, J. Lavaei are with the Department of Industrial Engineering and Operations Research, University of California, Berkeley. Email:
        {\tt\small \{yuhao\_ding, han\_feng, lavaei\}@berkeley.edu}}%
}

\begin{document}
\maketitle
\thispagestyle{empty}
\pagestyle{empty}

\begin{abstract}
This paper is concerned with numerically finding a global solution of constrained optimal control problems with many local minima. The focus is on the optimal decentralized control (ODC) problem, whose feasible set is recently shown to have an exponential number of connected components and consequently an exponential number of local minima. The rich literature of numerical algorithms for nonlinear optimization suggests that if a local search algorithm is initialized in an arbitrary connected component of the feasible set, it would search only within that component and find a stationary point there. This is based on the fact that numerical algorithms are designed to generate a sequence of points (via searching for descent directions and adjusting the step size), whose corresponding continuous path is trapped in a single connected component. In contrast with this perception rooted in convex optimization, we numerically illustrate that local search methods for non-convex constrained optimization can obliviously jump between different connected components to converge to a global minimum, via an aggressive step size adjustment using backtracking and the Armijio rule. To support the observations, we prove that from almost every arbitrary point in any connected component of the feasible set, it is possible to generate a sequence of points using local search to jump to different components and converge to a global solution. However, due to the NP-hardness of the problem, such fine-tuning of the parameters of a local search algorithm may need prior knowledge or be time consuming. This paper offers the first result on escaping non-global local solutions of constrained optimal control problems with complicated feasible sets.
\end{abstract}
\section{INTRODUCTION}

The linear-quadratic regulator (LQR) optimal control problem has been extensively studied in the past century~\cite{willemsLeastSquaresStationary1971, Anderson_1971}. A renewed interest in this classical topic is partially driven by tools in machine learning, where the successful applications of general optimization methods call for new theoretical analyses~\cite{fazelGlobalConvergencePolicy2018,ouyangControlUnknownLinear2017}. The behavior of more complex methods like policy gradient in reinforcement learning~\cite{rechtTourReinforcementLearning2018a} can also be understood in their application to linear-quadratic problems. They serve as a suitable baseline, because they admit well-known linear optimal solutions given by the Riccati equations~\cite{dullerud2013course} and an elegant parametrization of all sub-optimal solutions~\cite{Doyle1989}. Both properties, however, break down when we impose structures such as locality and delay on the controller~\cite{Witsenhausen1968}. 

The problem of finding an optimal controller subject to structural constraints is known as the optimal decentralized control (ODC) problem. ODC has been proved to be NP-hard~\cite{blondelNPHardnessLinearControl1997}, and an extensive research effort has been devoted to identifying structures or approximations that bypass the worse-case exponential complexity. It is known that the existence of stabilizing dynamic structured feedback is captured by the notion of decentralized fixed modes~\cite{Shih-HoWang1973}. When the system is spatially invariant~\cite{Bamieh2002}, hierarchical~\cite{salapakaStructuredOptimalRobust2004}, positive~\cite{Rantzer2015}, or quadratic invariant~\cite{rotkowitzConvexityOptimalControl2010}, ODC has a convex formulation. A System Level Approach~\cite{wangSystemLevelApproach2016} also convexifies ODC at the expense of working with a series of impulse response matrices. Various approximation ~\cite{Fattahi,Arastoo2015,matni2013heuristic} and convex relaxation techniques~\cite{Sojoudi2013, Fazelnia2014, Boyd1994} also exist in the literature. 

On the algorithmic side, nonlinear programming methods have been applied to instances of ODC to promote sparsity in controllers~\cite{Lin_TAC_2013}, or to approximate the optimal solution with prior constraints~\cite{Levine_1970, Wenk1980,Lin_TAC_2011}. Early works have been summarized in the survey~\cite{makilaComputationalMethodsParametric1987}, where various convergence rates have been discussed in the centralized controller case. In the decentralized case, the control literature lacks strategies to escape saddle points, or a guarantee of no spurious local optimum, or even an efficient initialization strategy that promotes convergence to a globally optimal solution. Those considerations in contrast have been extensively analyzed for many unconstrained problems in statistical learning~\cite{ge2016,Josz2018,lav_NIPS_2018_1}. We also mention that an interesting continuation method with risk-averse objective has been touched upon in~\cite{dvijothamConvexRiskAverse2014}. 

An often-overlooked aspect in ODC is that its feasible set can be disconnected. In fact, even for the simplest chain structure, the number of connected components may grow exponentially in the order of the system~\cite{Han_2019}. This means that since methods based on feasible-direction local search almost always assume connectivity in the underlying feasible set,  they may not be effective for finding a globally optimal solution to a general optimal decentralized control problem (because there are an exponential number of connected components, which implies at least the same number of local solutions and initializations in order to start in the correct connected component). 
More precisely, each connected component has a local solution and, therefore, feasible-direction local search methods should know which connected component has the global solution in order to start the iterations within that component. Note that such numerical algorithm generates a path from the initial point to the final stationary point being found by the algorithm. For convergence analysis, this path is often considered to be the discrete samples of a continuous path that is trapped within a single connected component where the algorithm is initialized.

In this work, we  show that numerical optimization algorithms are oblivious to the geometry of the feasible set and the discrete path of iterative points could potentially jump between connected components without realizing the existence of  discontinuity. We  also study a potentially infeasible-direction local search method, named  augmented Lagrangian~\cite{Lin_TAC_2011}, for which  the structural constraints can be violated at the beginning of the iterations but will be satisfied asymptotically as the number of iterations increases. This paper shows empirically that for constrained optimal control problems with many local minima, jumping between components is likely with random initializations and aggressive step-size rules. This allows jumps from the sub-optimal  component to the  globally-optimal component and vice versa. Moreover, we prove that a succession of jumps to the globally optimal component with descent directions is possible for almost all initializations. The phenomenon of jumping between connected components is appealing, though finding the correct step size could be challenging in general, due to the NP hardness of the problem. In summary, this work shows that unlike convex optimization where a small step size is used, an aggressive step size is the only viable method for escaping non-global local minima created by the discontinuity of the feasible set. 

A road-map for the remainder of the paper is as follows.
Notations and problem formulations are given in Section~\ref{sec:formulation}. 
Section~\ref{sec:local_search} gives an overview of two common local search algorithms, whose empirical performances are compared in Section~\ref{sec:simulation}. Section~\ref{sec:connectpath} proves that almost all initializations can be connected to the globally optimal component via descent directions. Concluding remarks are drawn in Section~\ref{sec:conclusion}.

\vspace{-2mm}

\section{PROBLEM FORMULATION AND PRELIMINARIES}\label{sec:formulation}
Consider the linear time-invariant (LTI) system
\begin{equation} \label{LTI}
\begin{aligned} 
    \Dot{x}(t)&=Ax(t)+Bu(t),\\
    y(t)&=Cx(t), 
\end{aligned}
\end{equation}
with an unknown initial state $x(0)=x_0$, where $x_0$ is treated as a random variable with a zero mean and the positive-definite covariance matrix $D_0$.
Consider also 
 the quadratic performance measure
\begin{align} \label{performance1}
\nonumber J(K)=\mathbb{E}_{x_0}\Big\{\int_0^{\infty}\Big[&x^\top(t)R_1x(t)+2x^\top(t)R_{12}u(t)\\
       &+u^\top(t)R_2u(t)\Big]dt\Big\}, 
\end{align}
where the matrix $\begin{bmatrix}{}
    R_1 & R_{12} \\
    R_{12}^\top & R_2 
\end{bmatrix}$ is positive smei-definite and $R_2$ is positive definite (the symbol $\mathbb{E}\{\cdot\}$ denotes the expectation operator).
We focus on the static case where the control input $u(t)$ is to be determined by a static output-feedback law $u(t)=-Ky(t)$. The objective is to design a  decentralized controller $K$ that belongs to a linear subspace $\mathcal{S} \subseteq \mathbb{R}^{m \times p}$, which models a user-defined decentralized control structure (note that $m$ and $p$ denote the dimensions of the input and output vectors, respectively). Let $\mathcal M$ denote the set of matrices $K$ for which all eigenvalues of $A-BKC$ are in the open left-half plane. The constrained optimal control problem of minimizing $J(K)$ over the feasible set $\mathcal M\cap \mathcal S$ is named optimal decentralized control (ODC) and can be formulated as
\begin{equation*}
\begin{aligned} \label{formulation1}
& \underset{K\in\mathcal M}{\text{minimize}}
& & \text{trace} (P(K)D_0 ) \\
& \text{subject to} & &  K \in \mathcal{S}
\end{aligned}
\tag{$P_1$}
\end{equation*}
where the matrix $P(K)$ denotes the closed-loop observability Gramian 
\begin{equation}
\begin{aligned}
    P(K)=\int_0^{\infty}\Big[&e^{(A-BKC)^\top t}[R_1-R_{12}KC-C^{\top}K^{\top}R_{12}^{\top}\\
            &+C^\top K^\top R_2KC]e^{(A-BKC) t}\Big]dt, 
\end{aligned}
\end{equation}
which can be equivalently obtained by solving the Lyapunov equation
\begin{equation}
\begin{aligned}
    &(A-BKC)^\top P+P(A-BKC)\\
    &=-(R_1-R_{12}KC -C^{\top}K^{\top}R_{12}^{\top}+C^\top K^\top R_2KC). 
\end{aligned}
\end{equation}
In  optimization~\eqref{formulation1}, since the open set $\mathcal M$ is a connected but a highly sophisticated set that cannot be efficiently characterized by algebraic equations, we regard it as the domain of the definition of the objective function $J(K)$. In contrast, even though the constraint $K\in\mathcal S$ causes the ODC to become NP-hard, it is a simple convex set and therefore we keep it as an explicit constraint in the problem. To handle the constraint $K\in\mathcal S$, one can impose it as a hard constraint or a soft constraint through a penalty function. To explain the latter approach, let $h:\mathbb R^{m \times p}\rightarrow \mathbb R$ be an arbitrary penalty function with the following properties: (i) $h(K)$ is continuous, (ii) $h(K)\geq 0$ for all $K\in\mathbb R^{m \times p}$, (iii) $h(K)=0$ if and only if $K\in\mathcal S$. Given a large positive constant $c$, the unconstrained counterpart of optimization~\eqref{formulation1}  is
\begin{equation*}
\begin{aligned} \label{formulation1_1}
& \underset{K\in\mathcal M}{\text{minimize}}
& & \text{trace} (P(K)D_0 ) +c h(K)
\end{aligned}
\tag{$P_1'$}
\end{equation*}
It is known that, under mild conditions,  \eqref{formulation1_1} can be used to find local minima of \eqref{formulation1} precisely for certain types of non-differentiable penalty functions (e.g., 1-norm penalty) and approximately with arbitrarily small errors for almost all differentiable penalty functions (e.g., quadratic penalty). To solve ODC numerically, we make the assumption that an initial feasible controller $K^0$ is available. This means the availability of a decentralized stabilizing controller $K^0\in\mathcal M\cap \mathcal S$ for   \eqref{formulation1} and a centralized stabilizing controller $K^0\in\mathcal M$ for   \eqref{formulation1_1}. Any descent algorithm generates a sequence of controllers $K^0,K^1,K^2,\ldots$.  The main difference between  \eqref{formulation1}  and \eqref{formulation1_1} is whether the constraint $K\in\mathcal S$ should be satisfied for all points of the sequence or only at its limit. The limit point of the sequence, if exists, could be a saddle point or a local minimum of the corresponding optimization problem. The work \cite{leeGradientDescentConverges2016} states that, under some conditions, the gradient descent algorithm with a random initialization and sufficiently small constant step sizes does not become stuck in a saddle point almost surely. However, since it is important to find a global solution of ODC, a question arises as to how many local minima \eqref{formulation1} or  \eqref{formulation1_1} has. The following result is a by-product of our recent work~\cite{Han_2019}.

\begin{lemma}
Suppose that $C$ has full row rank and  $\left[\begin{smallmatrix}
	R_1 & R_{12} \\ R_{12}^\top & R_2
\end{smallmatrix}\right]$ is positive definite. There are instances of the ODC problem for which the constrained optimization problem \eqref{formulation1} and its penalized counterpart \eqref{formulation1_1} both have an exponential number of local minima  (with respect to $n$) if $c$ is sufficiently large. 
\end{lemma}

\begin{proof} Consider any instance of the class of ODC problems given in \cite{Han_2019} with the property that the feasible set of the problem has an exponential number of connected components.  Due to the coercive property proven in Lemma~\ref{lem:lqrbounded} (stated later in the paper), each connected component must have a local minimum. Therefore,  \eqref{formulation1} has an exponential number of local minima. Let $\mathcal O$ denote the set of all local minima in any arbitrary connected component of the feasible set of ODC, and $\mathcal O(\epsilon)\subseteq\mathbb R^{m\times p}$ be the set of all points in the feasible set of \eqref{formulation1_1}  that are at most $\epsilon$ away from $\mathcal O$, for any given $\epsilon>0$. If  \eqref{formulation1_1} is numerically solved using gradient descent for an initial point in $\mathcal O(\epsilon)$, it follows from the proof technique given in \cite{Luenberger_book} that the algorithm will converge to a local minimum that is in the interior of $\mathcal O(\epsilon)$ and approaches $\mathcal O$  as $c$ goes to infinity. This implies that \eqref{formulation1_1}  has at least one local minimum corresponding to the set $\mathcal O$. Therefore, \eqref{formulation1_1} has an exponential number of local minima. 
\end{proof}

It should be noted that the work \cite{fazelGlobalConvergencePolicy2018} shows that if $c=0$, then \eqref{formulation1_1} has a single local solution (which should be global as well). However, the above result indicates the complexity added by softly penalizing the sparsity pattern of the controller. 

\subsection{Summary of Contribution}

Given the existence of many local minima for ODC in general, it is important to understand how effective local search algorithms are. These algorithms often have two parameters to design at every iteration: (i) descent direction, (ii) step size. Rooted in convex optimization, there is a large literature on how to design these two parameters to guarantee convergence to a solution. In particular, the existing solvers often use the backtracking technique to design a step size, which starts with a guess for the step size and then reduces it by a constant factor iteratively until an appropriate value is found. The initial guess for backtracking is often considered small since a large number does not offer any major benefits for convex problems. In this work, we show the contrary and prove that a large initial step for backtracking, which we name {\it aggressive local search}, has the ability to skip local minima by jumping from one connected component of the feasible set to another one. We first numerically illustrate this idea and then theoretically show that there exist values for the parameters (i) and (ii) of local search to guarantee convergence to a global solution from almost every feasible point. This positive result implies that local search does not necessarily become stuck even with a bad initialization, but finding the right parameters (i) and (ii) would be difficult in the worst case due to the NP-hardness of the problem. 

\vspace{-2mm}

\subsection{Algebraic Characterization of Structural Constraints}
Similar to ~\cite{Lin_TAC_2013} and ~\cite{Lin_TAC_2011}, we introduce a \textit{structural identity} matrix $I_{\mathcal{S}}$ of the linear subspace $\mathcal{S}$ to algebraically characterize the structural constraint $K\in \mathcal{S}$. The $(i,j)$-entry of $I_{\mathcal{S}}$ is defined as
\begin{align*}
[I_{\mathcal{S}}]_{ij}= \begin{cases} 1, \qquad \text{if $K_{ij}$ is a free variable,}\\
0, \qquad \text{if $K_{ij}=0$ is required.}
\end{cases}
\end{align*}   
Let $I_\mathcal{S}^c:=\mathbf{1}-I_\mathcal{S}$ be the structural identity of the complementary subspace $\mathcal{S}^c$, where $\mathbf{1}$ is the matrix with all its entries equal to one. One can write
\begin{align*}
K \in \mathcal{S} \Longleftrightarrow K \odot I_\mathcal{S}=K \Longleftrightarrow K \odot I_\mathcal{S}^c=0
\end{align*}
where $\odot$ denotes the entry-wise multiplication of matrices. The Formulation \eqref{formulation1} can be written as 
\begin{equation*}
\begin{aligned} \label{formulation2}
& \underset{K\in\mathcal M}{\text{minimize}}
& & \text{trace} (D_0 P(K)) \\
& \text{subject to}
& &  K \odot I_\mathcal{S}^c=0
\end{aligned}
\tag{$P_2$}
\end{equation*}

\vspace{-2mm}

\subsection{Inverse Optimal Control}
It is known in the context of inverse optimal control~\cite{inverse_TAC_1972} that any static state-feedback gain $K^{\text{opt}}$ is the unique minimizer of some quadratic performance measure \eqref{performance1} for all initial states. One such measure is
\begin{align} 
    \int_0^{\infty}\Big[(u(t)+K^{\text{opt}}Cx(t))^\top R_{2}(u(t)+K^{\text{opt}} Cx(t))\Big]dt.
\end{align}
Accordingly, we write $R_1$ and $R_{12}$ as 
\begin{equation} \label{inverse_weights}
R_1=C^\top {K^{\text{opt}}}^{\top}R_{2} K^{\text{opt}} C, \, \, R_{12}=C^\top {K^{\text{opt}}}^{\top} R_{2}. 
\end{equation}
If $K^{\text{opt}} \in \mathcal{S}$, the construction in \eqref{inverse_weights} ensures that $K^{\text{opt}}$ is the globally optimal controller in both the decentralized and the centralized settings. We will use this fact to conduct case studies later in this paper.



\vspace{-2mm}

\section{LOCAL SEARCH ALGORITHMS} \label{sec:local_search}
In this section, we give an overview of two optimization frameworks that have been applied to instances of ODC to deal with  structural constraints: the projection-based method~\cite{Lin_TAC_2013} and the augmented Lagrangian method~\cite{Lin_TAC_2011}.

We first derive the objective function's first and second derivatives. This will lead to the necessary optimality condition that can be exploited to develop local search algorithms to solve the ODC problem.
Applying the standard techniques~\cite{Levine_1970},~\cite{Rautert_1997} to the LTI system \eqref{LTI} with the quadratic performance measure \eqref{performance1}, we obtain the first- and second-order derivatives of $J$ as follows. 
\begin{prop}[]
The gradient of $J$ is given by
\begin{equation} \label{G_J}
   \nabla J(K)=2(R_2KC-R_{12}^\top-B^\top P)LC^\top, 
\end{equation}
where $L$ and $P$ are the controllability and observability Gramians of the closed-loop system, given by
\begin{equation} \label{FON-L}
\begin{aligned}
 \hspace*{-0.5cm}& (A-BKC)L+L(A-BKC)^\top=-D_0, 
  \end{aligned}\tag{$FON-L$}
\end{equation}
and
\begin{equation}\label{FON-P}
\begin{aligned}
    &(A-BKC)^\top P+P(A-BKC)\\
    &=-(R_1-R_{12}KC -C^{\top}K^{\top}R_{12}^{\top}+C^\top K^\top R_2KC). 
  \end{aligned}\tag{$FON-P$}
\end{equation}
\end{prop}
\begin{prop}[]
The second-order approximation of $J$ is determined by 
\begin{equation}
  J(K+\tilde{K}) \approx J(K) +\<\nabla J(K), \tilde{K}\>+\frac{1}{2}\<H_J(K,\tilde{K}),\tilde{K}\>, 
\end{equation}
where $H_J(K,\tilde{K})$ is 
\begin{equation} \label{H_J}
2\Big((R\tilde{K}C-B^\top \tilde{P})LC^\top +(R_2KC-R_{12}^\top-B^\top P)\tilde{L}C^\top \Big), 
\end{equation}
and $\tilde{L}$ and $\tilde{P}$ are the solutions of the following Lyapunov equations:
\begin{equation}\label{SON-L}
\begin{aligned}
 & (A-BKC)\tilde{L}+\tilde{L}(A-BKC)^\top=B\tilde{K}CL+(B\tilde{K}CL)^\top, 
  \end{aligned}
\end{equation}
\begin{equation}\label{SON-P}
\begin{aligned}
 &(A-BKC)^\top \tilde{P}+\tilde{P}(A-BKC)=(R_{12}^\top+B^\top P \\
 &-R_2KC)^\top\tilde{K}C+ \Big( (R_{12}^\top+B^\top P-R_2KC)^\top\tilde{K}C \Big)^\top. 
  \end{aligned}
\end{equation}
\end{prop}

Note that the notation $\<\cdot,\cdot\>$ used above is the inner product operator. To solve the constrained optimal control problem numerically, we can start with an initial stabilizing $K^0\in S$ and generate a descent stabilizing sequence \{$K^i$\} 
using the update $K^{i+1}=K^i+s^i\tilde{K}^i$, where $\tilde{K}^i \in \mathcal{S}$ is a descent direction determined by the first-order and possibly the second-order information, and $s^i$ is the step size.

\vspace{-2mm}

\subsection{Projection-based method}
Since the structural constraint $ K \odot I_\mathcal{S}^c=0$ is linear, we can project the gradient $\nabla J(K)$ and $H_J(K,\tilde{K})$ onto the linear subspace $\mathcal{S}$ to guarantee the satisfaction of the structural constraints. The projected gradient of $J$ can be expressed as 
\begin{equation} \label{pro_G_J}
    \nabla J(K) \odot I_\mathcal{S}=2\big((R_2KC-R_{12}^\top-B^\top P)LC^\top \big)\odot I_\mathcal{S}.
\end{equation}
Then, given $L$ and $P$,  the first-order optimality condition $ \nabla J(K) \odot I_\mathcal{S}=0$ is a linear equation involving an entry-wise product. 
Based on the first-order condition \eqref{pro_G_J}, the alternating method (the so-called Anderson-Moore or A-M method)~\cite{Anderson_1971} can be employed. Starting with a decentralized stabilizing controller $K\in \mathcal{S}$, this method alternates between solving the two Lyapunov equations \eqref{FON-L} and \eqref{FON-P}, and solving the linear equation \eqref{pro_G_J}. It is shown in~\cite{Rautert_1997} that the difference between two consecutive steps $K^{i+1}-K^i$ is a descent direction and therefore the alternating method will converge to a stationary point of \eqref{formulation2}. The advantage of this algorithm lies in its fast convergence compared to the gradient method~\cite{makilaComputationalMethodsParametric1987},~\cite{Rautert_1997}.

We next consider the second-order information. With the structural constraint $K\in \mathcal{S}$, the second-order approximation of $J$ can be expressed as 
\begin{equation} \label{pro_H_J}
J(K) +\<\nabla J(K)\odot I_\mathcal{S}, \tilde{K}\>+\frac{1}{2}\<H_J(K,\tilde{K})\odot I_\mathcal{S},\tilde{K}\>, 
\end{equation}
where $\nabla J(K)$ and $H_J(K,\tilde{K})$ are defined in \eqref{G_J} and \eqref{H_J}.
Based on the second-order information and its corresponding necessary optimality condition, Newton's method can be applied to determine the descent direction by minimizing the second-order approximation \eqref{pro_H_J} of the objective function with the structural constraints. To avoid inverting the large Hessian matrix explicitly, the conjugate gradient method can be employed to compute the Newton direction~\cite[Chapter~5]{NoceWrig06}.

Both descent directions described above can be combined with line search methods. The commonly applied backtracking with Armijo rule 
 selects $s^i$ as the largest number in $\{\bar{s}, \bar{s}\beta, \bar{s}\beta^2, ...\}$ such that $K^i+s^i\tilde{K}^i$ is stabilizing and 
\begin{equation}
J(K^i+s^i\tilde{K}^i)<J(K^i)+\alpha s^i \<\nabla J(K^i), \tilde{K}^i\> , 
\end{equation}
where $\alpha, \beta \in (0,1)$ and $\bar{s}$ is the initial step step size. Selecting a large value for $\bar s$ corresponds to aggressive local search. 



\vspace{-2mm}

\subsection{Augmented Lagrangian method}
Instead of forcing the sparsity constraint by projecting $\nabla J(K)$ and $H_J(K,\tilde{K})$ onto the subspace $\mathcal{S}$, the augmented Lagrangian method~\cite{Lin_TAC_2011} minimizes a sequence of unstructured problems. The augmented Lagrangian function for \eqref{formulation2} is given by
\begin{equation} \label{Aug_lag}
    \mathcal{L}_c(K,V)=J(K)+\<V,K \odot I_\mathcal{S}^c\>+\frac{c}{2}\| K \odot I_\mathcal{S}^c\|^2, 
\end{equation}
where the penalty weight $c$ is a positive scalar, $\| \cdot\|$ is the Frobenius norm, and the Lagrangian multiplier $V\in \mathcal{S}^c$ together with a local minimum of \eqref{formulation2}  is assumed to satisfy the second-order sufficient optimality conditions. 
The augmented Lagrangian method starts from an initial estimate of the Lagrangian multiplier $V^0$, and then alternates between minimizing $\mathcal{L}_c(K,V^i)$ with respect to the unstructured $K$ for fixed $V^i$:
\begin{equation*}
K^{i}= \argmin \mathcal{L}_c(K,V^i), 
\end{equation*}
and updating the Lagrangian multiplier:
\begin{equation*}
V^{i+1}=V^i+c(K^i \odot I_\mathcal{S}^c ).
\end{equation*}
To ensure convergence and avoid the ill-conditioning in minimizing $\mathcal{L}_c(K,V)$, a practical scheme is to update the penalty weight as $
c^{i+1}=\gamma c^i $ with $\gamma >1$
until it reaches a certain threshold value $\tau$.
The augmented Lagrangian method terminates as soon as $\| K \odot I_\mathcal{S}^c\|<\epsilon$ is reached.

Similar with the projection-based method, we can use the alternating method or Newton's method combined with the Armijo rule to solve the unconstrained augmented Lagrangian function  $\mathcal{L}_c(K,V^i)$. 
The gradient of $\mathcal{L}_c(K,V^i)$ can be expressed as
\begin{equation} \label{pro_G_lag}
    \nabla \mathcal{L}_c(K,V^i) = 2(R_2KC-R_{12}^\top-B^\top P)LC^\top  +V^i+c(K \odot I_{\mathcal{S}}^c)
\end{equation}
Then, given $L$ and $P$,  the first-order optimality condition $ \mathcal{L}_c(K,V^i)=0$ is a linear equation involving an entry-wise product. 
Based on the first-order condition \eqref{pro_G_J}, the alternating method solves the two Lyapunov equations \eqref{FON-L} and \eqref{FON-P}, and then solves the linear equation \eqref{pro_G_lag}. It is proven in~\cite{Lin_TAC_2011} that the difference between two consecutive steps $K^{i+1}-K^i$ is also a descent direction for the unconstrained augmented Lagrangian function, 
thereby ensuring the convergence
to a stationary point of $\mathcal{L}_c(K)$. 
Newton's method 
can also be applied to minimize $\mathcal{L}_c(K)$ since it is well-suited for ill-conditioned $\mathcal{L}_c(K)$ when the penalty weight $c$ becomes large~\cite[Section~5.2]{bertsekas_2016nonlinear}. To do so, we only need to minimize the second-order approximation of $\mathcal{L}_c(K)$:
\begin{equation} \label{pro_H_Lag}
\mathcal{L}_c(K) +\<\nabla \mathcal{L}_c(K), \tilde{K}\>+\frac{1}{2}\<H_{\mathcal{L}}(K,\tilde{K}),\tilde{K}\>, 
\end{equation}
where $H_{\mathcal{L}}(K,\tilde{K})$ is 
\begin{equation*} \label{H_Lag}
2\Big((R\tilde{K}C-B^\top \tilde{P})LC^\top +(R_2KC-R_{12}^\top-B^\top P)\tilde{L}C^\top \Big) +c\tilde{K}, 
\end{equation*}
and then use the conjugate gradient method to compute the Newton direction.

\vspace{-2mm}

\section{Case Studies} \label{sec:simulation}
In this section, we test the methods of Section \ref{sec:local_search} on examples in~\cite{Han_2019}, where the feasible set of the constrained optimal control problem has an exponential number of connected components and consequently an exponential number of local minima. 
Consider the LTI system in \eqref{LTI} such that $A$ is of the form
\begin{align}\label{eq:aeg}
A = \begin{bmatrix}
    f_1 + \epsilon & f_2 & 0 & \cdots & \cdots & 0\\
    -h_2  &  \epsilon  & f_3 & \ddots  && \vdots \\
      0    & -h_3  &  \epsilon & f_4  & \ddots & \vdots\\ 
       \vdots    &\ddots & \ddots  & \ddots  & \ddots  & 0 \\ 
       \vdots     &&\ddots& -h_{n-1} & \epsilon & f_n \\ 
        0    &\cdots&\cdots & 0 & -h_n & \epsilon
\end{bmatrix},
\end{align}
where $\epsilon>0$, $f_1<0$, and $(-1)^i(f_i-h_{i+1})>0$ for $i=2,\ldots,n$.  
Let $B\in\mathbb{R}^{n\times n}$, $C\in\mathbb{R}^{n \times n}$, $D_0\in\mathbb{R}^{n \times n}$ and $I_\mathcal{S}\in\mathbb{R}^{ n\times n}$ be of the form
\begin{align} \label{eq:beg}
B &= \left[\begin{array} {cccc}
     0 & 1 \\
     -1 &  \ddots& \ddots \\ 
     &\ddots & 0 & 1 \\
      &  & -1 & 0\\ 
\end{array}\right], \ C = I,\ D_0= I, \ I_\mathcal{S}=I
\end{align}
It is proven in~\cite{Han_2019} that for a small enough $\epsilon\geq0$, the set
\begin{equation*}
  \mathcal{K}=  \{K: A-BKC \text{ is stable/Hurwitz, } K \in \mathcal{S} \} 
\end{equation*}
has at least $F_n$ connected components, where $F_0=1, F_1=1, F_{i+2}=F_{i+1}+F_i$ for $i=0, 1, \dots$ is the Fibonacci sequence. Note that $F_n$ grows exponentially in $n$.

\vspace{-2mm}

\subsection{Performance of projection-based method}
Although the observations to be discussed next are also valid for large values of $n$, we restrict the simulations to $n=3$ so that the results can be visualized. Consider the third-order system ($n=3$) with $f_1=-1, f_2=h_2=10, f_3=h_3=1$ and
\begin{equation} \label{n3_no_epislon}
\begin{aligned}
K_c =20I,  \quad
R_2=  
\left[\begin{array} {cccc}
     20 & 1 & -1 \\
     1 & 5 & 2\\ 
     -1 & 2 & 2 \\
\end{array}\right], \quad
\epsilon=0
\end{aligned}
\end{equation}
where $K_c$ is the optimal centralized controller, and $R_1$ and $R_{12}$ are accordingly computed by \eqref{inverse_weights}. Assume that the set $\mathcal{S}$ consists of only purely diagonal matrices, meaning that a decentralized controller is to be designed. The feasible set of the ODC problem has $3$ connected components with no margin between the closures of the components (as shown in Fig. \ref{fig:step_noe}). The parameters of the Armijo rule are set as $\bar{s}=1, \beta=0.5, \alpha=10^{-2}$ and the stopping criterion is $\|\nabla J(K) \odot I_\mathcal{S}\|<10^{-3}$. The initial points are randomly sampled among the structured stabilizing controllers. 

\subsubsection{Jumping bewteen connected components} 
Some of the convergence results from random initializations are summarized in Table \ref{table: results_proj} and some of the trajectories are plotted in Fig. \ref{fig:step_noe}. We use $D_j(x)$ to denote the diagonal matrix where the vectorized diagonal elements are the vector $x$ and the subscript $j$ is the index of the connected component that the corresponding feedback gain belongs to. For example, $D_1(40,40,40)$ represents the diagonal feedback gain with the diagonal entries $40,40,40$ in the connected component $1$. We also use the notation $K^+$ to denote any locally optimal solution and $K_j$ to denote any locally optimal solution in the component $j$. In this example, we have $K_1=K_c$, $K_2=D_2(6.06,-3.16,-0.63)$ and $K_3=D_3(6.48,6.46,3.02)$. Note that $K_c$ is by design the best centralized controller and since it is already diagonal, it is the globally optimal solution of ODC.

From Table~\ref{table: results_proj}, we can see that a jump can occur from the globally optimal component to the sub-optimal component and vice versa. Therefore, on the one hand, the projection-based method can not guarantee the convergence to the globally optimal solution even if initialized in the the globally optimal component. On the other hand, even if initialized in the  sub-optimal component, the projection-based local search is still likely to find the globally optimal solution by jumping to the globally optimal component. This observation also supports the conclusion in Section \ref{sec:connectpath} that except for a set of measure zero, all initial points of the decentralized LQR problem can be connected to the globally optimal decentralized controller via a path that involves only descent directions. 

\begin{table}
\caption{Jump behavior for the projection-based method with $\epsilon=0$}
\centering 
\vspace*{-0.2cm}
\begin{tabular}{ccccc}
\hline
\multirow{2}{*}{$K^0$} & \multicolumn{2}{c}{A-M} & \multicolumn{2}{c}{Newton} \\ \cline{2-5} 
& $K^+$ & $J(K^+)$ & $K^+$ & $J(K^+)$ \\ \hline
$D_1(40,40,40)$ & $K_c$ & 0  & $K_c$  & 0 \\
$D_1(97,80,121)$ & $K_c$  & 0  &  $K_3$ & $7357.5$  \\
$D_1(135,126,171)$ & $K_c$  & 0  &  $K_3$ & $7357.5$ \\
$D_2(0,0,0)$ & $K_2$  &  $16237.0$  & $K_2$ &  $16237.0$ \\ 
$D_2(-54,-26,-41)$ & $K_2$  &  $16237.0$ & $K_c$& $0$\\ 
$D_2(-34,-12,-5)$ & $K_c$  & $0$ & $K_2$ & $16237.0$ \\ 
$D_3(-10,5,10)$ & $K_3$ & $7357.5$ & $K_3$  & $7357.5$ \\ 
$D_3(-19,5,6)$ & $K_c$ &  $0$  & $K_3$  & $7357.5$ \\
$D_3(-25, 6,5 )$ & $K_c$ & $0$  & $K_3$  & $7357.5$ \\ 
\hline
\end{tabular}
\label{table: results_proj}
\end{table}

\subsubsection{Strict separation and exponential number of connected components}
We next consider the same system in \eqref{n3_no_epislon} but with $\epsilon=0.05$. In this case, the connected components will be strictly separated (as shown in Fig. \ref{fig:step_eee}). As $\epsilon$ increases, the disconnected components become more separated~\cite{Han_2019}. Although the jump between the connected components still occurs, the projection-based local search methods is more likely to become stuck in the connected component that contains the initial point since the step size is not adaptively designed. Table \ref{table: compare_epi} compares the number of jumps from the sub-optimality components to the globally optimality component in $10,000$ random initialization trials for different values of $\epsilon$. With the slightly abuse of notation, we use $D_2$ and $D_3$ to denote the component $2$ and $3$ that are the sub-optimal components.

\begin{table}
\caption{Number of jumps from sub-optimality components to globally optimality component with different $\epsilon$.}
\centering 
\vspace*{-0.2cm}
\begin{tabular}{ccccccc}
\hline
\multirow{2}{*}{} & \multicolumn{3}{c}{A-M} & \multicolumn{3}{c}{Newton} \\ \cline{2-7} 
 & $\epsilon=0$ & $\epsilon=0.05$ & $\epsilon=0.1$ & $\epsilon=0$ & $\epsilon=0.05$& $\epsilon=0.1$ \\ \hline
$D_2$ & $105$ & $29$ & $7$ & $1841$ & $586$ &$101$ \\
$D_3$ & $143$ & $135$ & $102$ & $0$ & $0$ & $0$ \\ \hline
\end{tabular}
\label{table: compare_epi}
\end{table}

We comment that in this example, as the dimension $n$ increases, the number of connected components increases exponentially. Therefore, the likelihood of jumping from the sub-optimal component to the globally optimal component is slim with a small step size.

\subsubsection{Aggressive step size}
Table \ref{table: compare_step} compares the number of jumps from the sub-optimal components $D_2$ and $D_3$ to the globally optimal component in $10,000$ random initialization trials for different initial step sizes $\bar{s}$ and $\beta$. The trajectories initialized at the same points but with the different step size parameters 
are plotted in Fig. \ref{fig:step_noe} and Fig. \ref{fig:step_eee}. From the above numerical results, we can see that the projection-based local search method would fail if the step size is restricted to be  small  in the backtracking of the Armijo rule. In  convex optimization as well as those nonlinear problems with a connected feasible region, the common practice is to consider the step size in descent algorithms to be small to guarantee the convergence to a locally/globally optimal solution. Note that the upper bound on the step size is often considered to be less than a constant factor of the inverse of the Lipschitz constant of the objective function \cite{bertsekas_2016nonlinear}. 
But, using a small step size in the above example, we almost always obtain a non-global local solution whenever we initialize the algorithm in a sub-optimal component. However, since numerical algorithms are oblivious to the geometry of the feasible sets, they can be deceived by selecting large step sizes to make them jump from one connected component to another one without realizing discontinuity.

\begin{figure}
	\begin{center}
	\subfigure[Alternating method\label{fig:pro_AM}]
{\includegraphics[width=3.5in,trim=0mm 5mm 0mm 0mm]{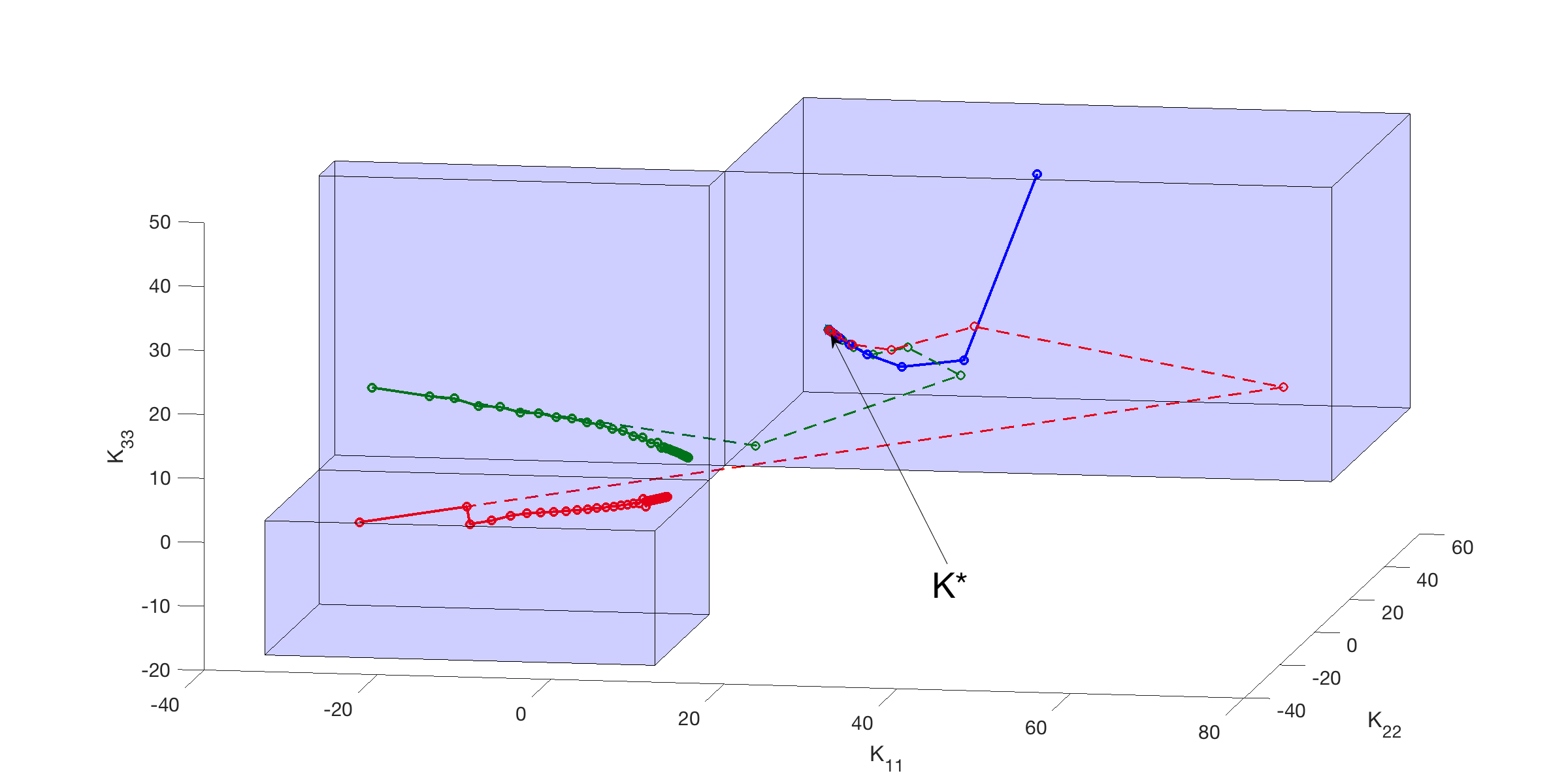}}
	\subfigure[Newton's method \label{fig:pro_AM_cen}]
	{\includegraphics[width=3.5in,trim=0mm 5mm 0mm 3mm]{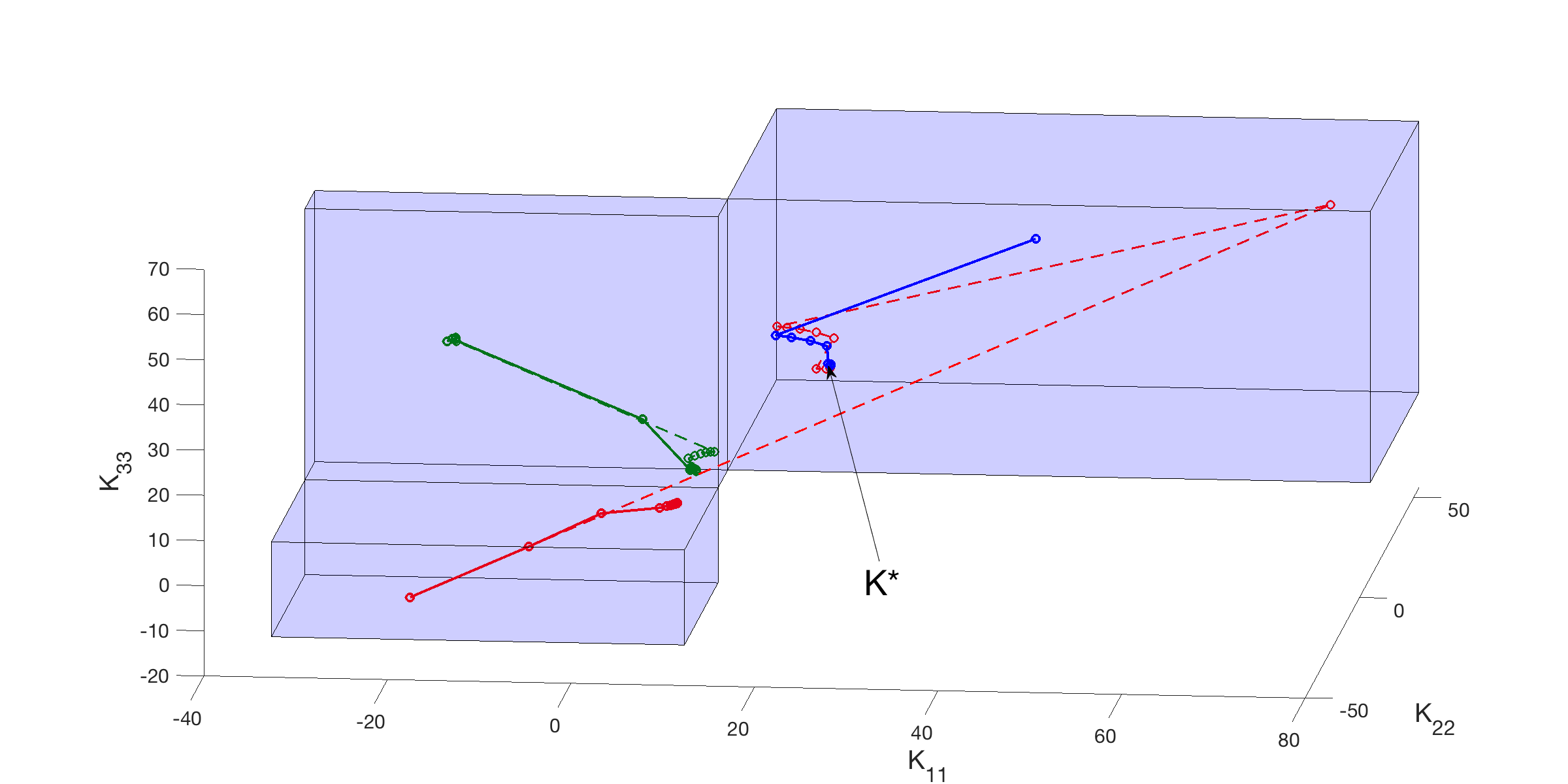}}
		\caption{\footnotesize The trajectory of the sequence $\{K^i\}$ obtained by the projection-based method with $\epsilon=0$. The cubes are the connected components for the decentralized controller. Solid lines correspond to $s^0=1, \beta=0.5$ and the dashed lines correspond to $s^0=5, \beta=0.9$. (a): The blue, red and green line are initialized at the points $D_1(40,40,40)$, $D_2(-28.3,-8.9,-4.4)$ and $D_3(-30.1,7.1,12.6)$, respectively. (b): The blue, red and green line are initialized at the points $D_1(40,40,40)$, $D_2(-21.9,-12.9,-18.3)$ and $D_3(-20,6,30)$, respectively. Note that $K_{11}$, $K_{22}$ and $K_{33}$ denote the diagonal entries of the controller $K$.}
	 \label{fig:step_noe} \normalsize
	\end{center}
	\vspace{-0.2cm}
\end{figure}

\begin{figure}
	\begin{center}
	\subfigure[Alternating method\label{fig:pro_AM}]
{\includegraphics[width=3.5in,trim=0mm 5mm 0mm 0mm]{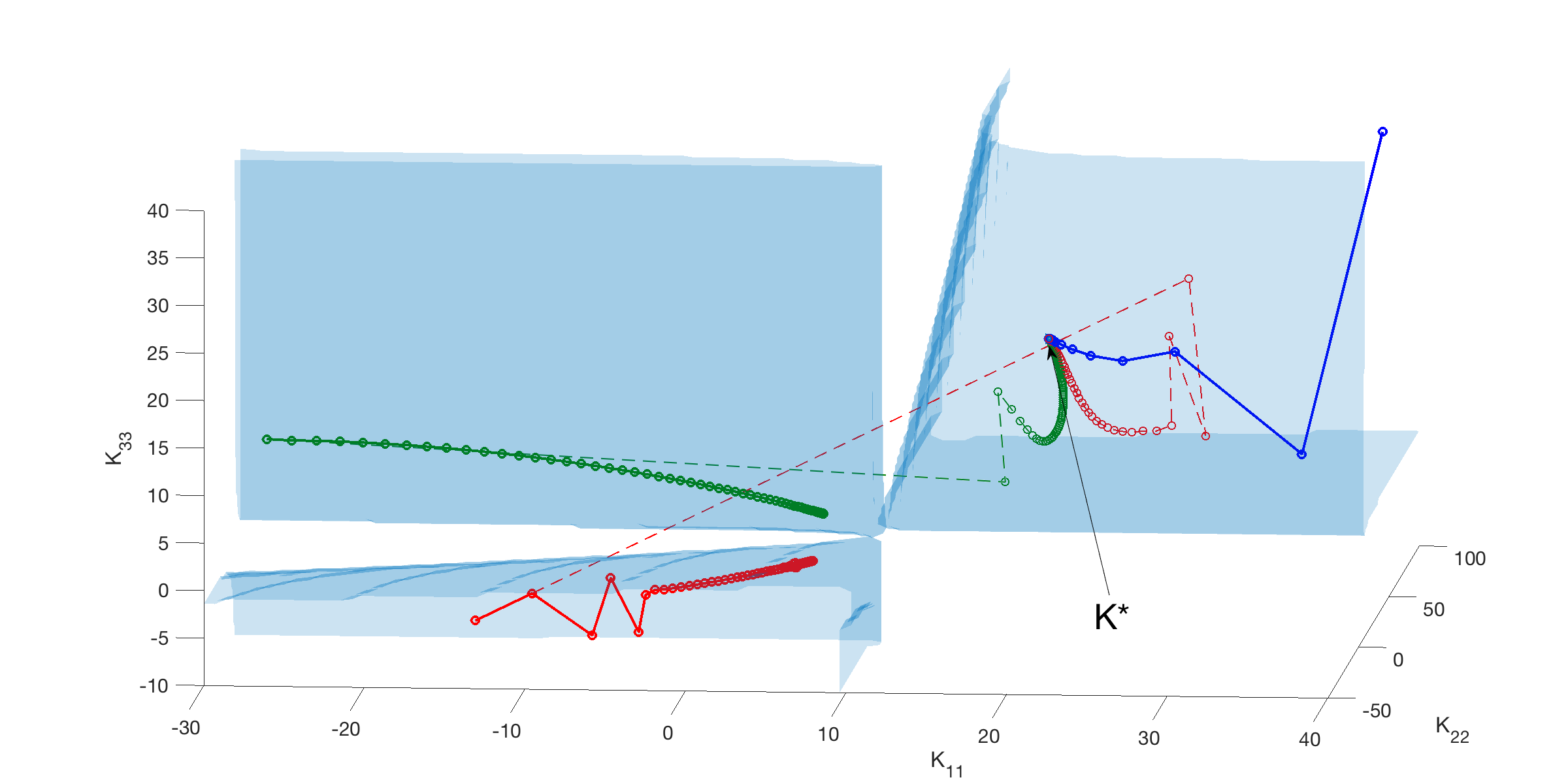}}
	\subfigure[Newton's method \label{fig:pro_AM_cen}]
	{\includegraphics[width=3.5in,trim=0mm 5mm 0mm 3mm]{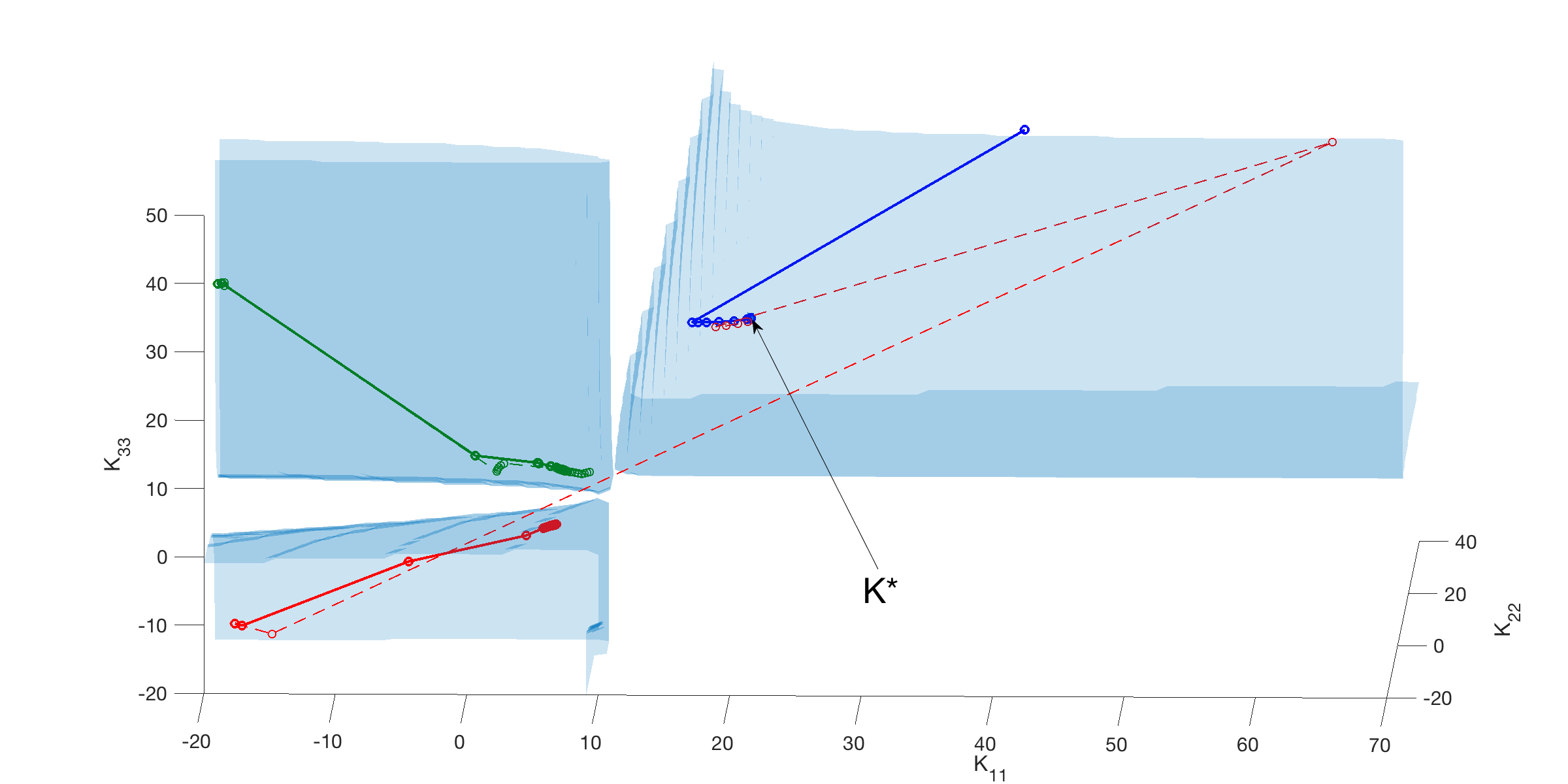}}
		\caption{\footnotesize The trajectory of the sequence $\{K^i\}$ obtained by the projection-based method with $\epsilon=0.05$. The blue regions are the connected components for the decentralized controller. Solid lines correspond to $s^0=1, \beta=0.5$ and dashed lines correspond to $s^0=5, \beta=0.9$. (a): The blue, red and green line are initialized at the points $D_1(40,40,40)$, $D_2(-14.2,-20.7,-6)$ and $D_3(-28.2, 6.1, 10)$, respectively. (b): The blue, red and green line are initialized at the points $D_1(40,40,40)$, $D_2(-18.3,-3.6,-16)$ and $D_3(-20,6,30)$, respectively. }
	 \label{fig:step_eee} \normalsize
	\end{center}\vspace{-0.2cm}
\end{figure}

\begin{table}
\caption{Number of jumps from sub-optimal components to globally optimal component with different $s^0, \beta$. }
\centering 
\vspace*{-0.2cm}
\begin{tabular}{ccccc}
\hline
\multirow{2}{*}{} & \multicolumn{2}{c}{A-M} & \multicolumn{2}{c}{Newton} \\ \cline{2-5} 
 & $s^0=1$ & $s^0=5$ & $s^0=1$ & $s^0=5$ \\ & $\beta=0.5$ & $\beta=0.9$  & $\beta=0.5$ & $\beta=0.9$ \\\hline
 $D_2$ & $29$ & $178$ & $586$ & $6674$ \\
$D_3$ & $136$ & $169$ & $0$ & $0$ \\  \hline
\end{tabular}
\label{table: compare_step}
\end{table}



\vspace{-2mm}

\subsection{Performance of augmented Lagrangian method}

We consider the third-order system given in \eqref{eq:aeg} and \eqref{eq:beg} with $f_1=-1, f_2=h_2=2, f_3=h_3=1$ and
\begin{equation} \label{n3_no_epislon_alm}
\begin{aligned}
K^c=  
\left[\begin{array} {cccc}
     6 & -10 &0 \\
     0 & 2 & -10\\ 
     4 & 0 & 0 \\
\end{array}\right], 
R_2= I,
\epsilon=0.
\end{aligned}
\end{equation}
where $R_1$ and $R_{12}$ are accordingly computed by \eqref{inverse_weights}. 
Here, the parameters of the augmented Lagrangian method are set as $V^0=0, c^0=10, \gamma=3, \tau=10^5$ and the parameters of the Armijo rule are set as $s^0=1, \beta=0.5, \alpha=10^{-2}$. The stopping criterion for the augmented Lagrangian method is $\|K \odot I_\mathcal{S}\|<10^{-4}$ and the stopping criterion for minimizing the unconstrained augmented Lagrangian function is $\|\nabla L_c(K)\|<10^{-2}$.

Since it is generally NP-hard to solve the ODC problem~\cite{Witsenhausen1968},~\cite{Blondel2000}, there is no efficient method to find a globally optimal decentralized controller with guarantees. However, from the fact that a random initialization in our simulation yields at least 2 local solutions with different objective values, we can still conclude that the augmented Lagrangian method with the random initialization fails to find the globally optimal decentralized controller. Some of the convergence results are summarized in Table~\ref{table: results_alm_10}. Here, local optimal solutions are $K_1=D_1(6.31,6.10,3.34)$, $K_2=D_2(0.69,-0.12, -0.34)$ and $K_3=D_3(0.69,-0.12, -0.34)$ and initial points are
\begin{equation*}
K^{01} \hspace{-0.08cm}= \hspace{-0.1cm} \left[\begin{array} {cccc}
172 & -260 & 42\\
130 & 184& -130\\
352 & 0 &-140\\
\end{array}\right]\hspace{-0.05cm}, 
K^{02}\hspace{-0.08cm}=\left[\begin{array} {cccc}
28 & -18.2 & 31 \\
9 & -6 & -9\\ 
18 & 0 & 40 \\
\end{array}\right]. 
\end{equation*}
\vspace{-0.3cm}

\begin{table}[ht]
	\caption{Numerical results for the augmented Lagrangian method with $c^0=10$}
	\centering 
	\vspace*{-0.2cm}
	\begin{tabular}{ccccc}
		\hline
		\multirow{2}{*}{$K^0$} & \multicolumn{2}{c}{Alternating} & \multicolumn{2}{c}{Newton} \\ \cline{2-5} 
		& $K^+$ & $J(K^+)$ & $K^+$ & $J(K^+)$ \\ \hline
		$K_c$ & $K_2$ & $332.5$  & $K_2$  & $332.5$  \\
		$K^{01}$ & $K_1$ & $454.3$ & $K_2$  & $332.5$  \\ 
		$K^{02}$ & $K_2$  &  $332.5$  & $K_1$ & $454.3$\\
		\hline
	\end{tabular}
	\label{table: results_alm_10}
\end{table}
\vspace{-0.3cm}


\subsubsection{Locally strong convexity and aggressive step size}
Table \ref{table: diff_c0} compares the number of convergences to $K_1$, $K_2$ and $K_3$ respectively in $700$ random initialization\footnote{Here, we only randomly sample the diagonal elements of $K\in \mathbb{R}^{n\times n}$ and then solve the feasibility problem such that the eigenvalues of $A-BKC$ are all in the open left half plane} trials for different initial penalty weight $c^0$. 

Penalty methods like the augmented Lagrangian method, which allow the violation of the structure constraints during the iterations, seem more likely to overcome the discontinuity of the feasible region than the projection-based method. However, the locally strong convexity introduced by the augmented Lagrangian function makes the local search more sensitive to the initial point. That is, as the initial penalty weight $c^0$ increases, the locally strong convexity near the subspace $I_\mathcal{S}$ also increases, which tends to attract the initial point to its closet local solution. Therefore, for these problems with a disconnected feasible region, the augmented Lagrangian method is not robust and can easily become stuck in a local solution. To overcome the locally strong convexity associated with the sub-optimal component, an aggressive step size is desirable.

\begin{table}[ht]
	\caption{Numerical results for the augmented Lagrangian method with different $c^0$}
	\centering 
	\vspace*{-0.2cm}
	\begin{tabular}{ccccccc}
		\hline
		\multirow{2}{*}{} & \multicolumn{3}{c}{Alternating} & \multicolumn{3}{c}{Newton} \\ \cline{2-7} 
		& $K_1$ & $K_2$ & $K_3$ & $K_1$ & $K_2$ & $K_3$ \\ \hline
		$C^0=10$ & $2$ & $698$ & $0$ & $18$ & $682$ & $0$  \\
		$C^0=50$ &  $531$ & $169$ & $0$ & $325$ & $375$ & $0$ \\
		$C^0=5000$ & $543$ &  $154$ & $3$ & $328$ & $363$ & $9$  \\ \hline
	\end{tabular}
	\label{table: diff_c0}
\end{table}



\section{Path to the globally optimal solution }\label{sec:connectpath}
In this section, we show that except for a set of measure zero, all initial stabilizing points of the decentralized LQR problem can be connected to the globally optimal decentralized controller via a path that involves only descent directions. The proof requires the result below on convergence to local minimizers. Given a twice continuously differentiable function $J(K)$, its stationary point $K^+$ solves $\nabla J(K^+)=0$. The function $J$ is said to satisfy \emph{strict saddle property}~\cite{ge2015escaping} if each critical point $K^+$ of $J$ is either a local minimizer or a ``strict saddle'', that is, $\nabla^2 J(K^+)$ has at least one strictly negative eigenvalue. 

\begin{lemma}[\cite{leeGradientDescentConverges2016}]\label{lem:convergelocal}
	If $f: \bR^d \to \bR$ is twice continuously differentiable and satisfies the strict saddle property, then gradient descent with a random initialization and sufficiently small constant step sizes converges to a local minimizer or negative infinity almost surely. 
\end{lemma}

To apply the lemma above
, we first 
show that the LQR problem has a certain structure that disallows the locally optimal stabilizing $K$ to have arbitrary magnitude. 

\begin{lemma}\label{lem:lqrbounded}
Consider the decentralized LQR problem in \eqref{formulation1}. Suppose that $C$ has full row rank,  $\left[\begin{smallmatrix}
	R_1 & R_{12} \\ R_{12}^\top & R_2
\end{smallmatrix}\right]$ is positive definite, and $K\in\mathcal{S}$ is stabilizing. Then, 
$
J(K) \to \infty
$
whenever $\|K\|_2 \to \infty$ or when $K$ approaches the boundary of the set of stabilizing controllers. Therefore, any descent method yields a bounded sequence of stabilizing controllers. 
\end{lemma} 

\begin{proof}
 We have 
 \begin{align*}
 P(K) = \int_0^\infty e^{t(A-BKC)^\top}\hat R(K) e^{t(A-BKC)} \text{d}t, 
 \end{align*}
 where 
 \[ \hat R (K) = R_1 - R_{12}KC - C^\top K^\top R_{12} + C^\top K^\top R_2 KC. \]
When $K$ is stabilizing, $P(K)$ is well-defined. As $K$ approaches a finite $K_\dagger$ on the boundary of the set of stabilizing controllers, we show that $\|P(K)\|_2 \to \infty$. By assumption, the symmetric matrix $\hat R(K)$ in the integral is positive definite, because it can be written as
\begin{align*}
\hat R(K_\dagger)= \begin{bmatrix}
	I  & - C^\top K_\dagger^\top
\end{bmatrix} \begin{bmatrix}
	R_1 & R_{12} \\ R_{12}^\top & R_2
\end{bmatrix} \begin{bmatrix}
	I \\ - K_\dagger C^\top 
\end{bmatrix}. 
\end{align*}
Therefore, its minimum eigenvalue $\lambda_{\min}(\hat R(K_\dagger)) > 0$, and when $K$ is close to $K_\dagger$, $\hat R(K) \succeq \frac12 \lambda_{\min}(\hat R(K_\dagger)) I $. We make the estimate 
\begin{align*}
\text{trace}(P(K)\!) & \hspace{-0.1cm} \geq \hspace{-0.1cm} \frac12 \lambda_{\min}(\hat R(K_\dagger)\!) \!\!\! \int_0^\infty\!\!\! \hspace{-0.22cm}\text{trace}\hspace{-0.08cm}\left(\! e^{t(A-BKC)^\top} \hspace{-0.1cm} \! e^{t(A-BKC)}\!\right) \hspace{-0.05cm}dt  \\ 
& \geq \frac12 \lambda_{\min}(\hat R(K_\dagger)) \int_0^\infty \|e^{t(A-BKC)} \|_2^2 dt  \\ 
& = \frac12 \lambda_{\min}(\hat R(K_\dagger)) \int_0^\infty e^{2t \cdot \text{spabs}(A-BKC)} dt, 
\end{align*}
where $\text{spabs}(\cdot)$ denotes the spectral abscissa (maximum real part of the eigenvalues). The estimate above shows that $\text{trace}(P(K))\to \infty$ as $K$ approaches $K_\dagger$ from the stabilizing set, hence $J(K) = \text{trace}(P(K)D_0) \geq \text{trace}(P(K)) \lambda_{\min}(D_0)$ also approaches infinity. 

In case $\|K\|_2 \to \infty$ from the stabilizing set, we use the fact that $P(K)$ is the unique solution to the equation 
\begin{align*}
 (A-BKC)^\top P + P (A-BKC) + \hat R(K) = 0.
\end{align*}
It follows from the triangle inequality that 
\begin{align*}
 & \lambda_{\min}(R_2) \sigma_{\min}^2(C)\|K\|_2^2  
\leq \|C^\top K^\top R_2 KC\|_2 \\ 
&  \leq 2  \|A-BKC\|_2\| P \|_2 +  \|R_1\|_2 + 2 \|R_{12}\|_2\| K\|_2 \|C\|_2  \\ 
& \leq 2 (\|A\|_2 + \|B\|_2\|K\|_2\|C\|_2)\| P \|_2 + \\ 
& \qquad \qquad \|R_1\|_2 + 2 \|R_{12}\|_2\| K\|_2 \|C\|_2. 
\end{align*}
where $\sigma_{\min}(C)$ is the minimum singular value of $C$. Therefore, 
\begin{align*}
& \| P \|_2 \geq  \\ 
& \frac{\lambda_{\min}(R_2)\sigma_{\min}^2(C) \|K\|_2^2 - \|R_1\|_2 - 2 \|R_{12}\|_2\| K\|_2 \|C\|_2}{2 (\|A\|_2 + \|B\|_2\|K\|_2\|C\|_2) }. 
\end{align*}
Hence, $\|P(K)\|_2\to \infty$ as $\|K\|_2 \to \infty$ inside the stabilizing set. Similarly $J(K) = \text{trace}(P(K)D_0) \geq \|P(K)\|_2 \lambda_{\min}(D)$ also approaches infinity. 
\end{proof}

Lemma~\ref{lem:lqrbounded} guarantees existence of a locally optimal decentralized controller in any connected component of the stabilizing set. Next, we show that around any strict local minimum of $J(K)$, there exists a controller from which a descent direction point towards a neighborhood of the globally optimal controller. 

\begin{lemma} \label{lem:find-global-direction}
	Suppose that $K^+$ is a strict local minimum of $J(K)$, and $K^+$ is not equal to a globally optimal solution $K^*$. Then, for all $\delta_0>0$, there exist a stabilizing $\hat K^+$ with $\|\hat K^+ - K^+\| \leq \delta_0$ and a number $\delta_1>0$, such that for all stabilizing $\hat K^*$ with $\|\hat K^* - K^*\| \leq \delta_1$, the direction $\hat K^* - \hat K^+$ is a descent direction at $\hat K_0$. 
\end{lemma}
\begin{proof}
	Since $K^+$ is a strict local minimum, there is a number $\delta_0'\in (0, \delta_0)$ such that when $\|K-K^+\| \leq \delta_0'$, we have $\nabla^2 J(K) \succ 0$. This implies that $f$ is strongly convex in this $\delta_0'$-neighborhood of $K^+$, that is,  whenever  $\|K-K^+\| \leq \delta_0'$, it holds that 
	\begin{align}
		\langle - \nabla J(K), K^+\hspace{-0.07cm} -\hspace{-0.07cm} K\rangle &\hspace{-0.07cm}=\hspace{-0.07cm} \langle \nabla J(K^+) - \hspace{-0.07cm} \nabla J(K), K^+\hspace{-0.07cm} - K\rangle \hspace{-0.07cm} > \hspace{-0.07cm} 0. \label{eq:local-strongly}
	\end{align}
    Especially, for $\hat K^+$ with  $\| \hat K^+ - K^+\| \leq \delta_0'$ and $K^+ - \hat K^+= \beta(K^* - \hat K^+)$ for some $\beta>0$, \eqref{eq:local-strongly} implies that $K^* - \hat K^+$ is a descent direction at $\hat K^+$. By continuity, there is a $\delta_1>0$ such that for all $\|\hat K^* - K^*\| \leq \delta_1$, $\hat K^* - \hat K^+$ is also a descent direction.
\end{proof}

\begin{remark}
When $\nabla J(K)$ is smooth, the claim of Lemma~\ref{lem:find-global-direction} can be strengthened. At a non-degenerate zero $K^+$ of the vector field $\nabla J(K)$, the index of the vector field is nonzero. Within the set of stabilizing controllers, a suitably small lower-level set of $J$ around $K^+$can be regarded as a manifold $X$ with boundary; its Gauss map to the unit sphere will have a non-zero degree~\cite[\S6]{milnorTopologyDifferentiableViewpoint1997}. Sard's theorem ~\cite{milnorTopologyDifferentiableViewpoint1997} implies that almost all directions in the unit sphere are achievable by some gradient of $J$ at the boundary of $X$. When $X$ is small, the direction $K^*-K^+$ is not so different from $K^*-K$; hence almost all points in a neighborhood of $K^*$ can be made arbitrarily close to some ray $K - \alpha \nabla J(K), \alpha>0$ with a suitable $K\in X$.  
\end{remark}

\begin{theorem}
	Consider a decentralized LQR problem with the same assumption as in Lemma~\ref{lem:lqrbounded}. Suppose that $J(K)$ satisfies the strict-saddle property and its local minima are all strict. Then, except for a set of measure zero, from every initial stabilizing controller
	, there is a path to a globally optimal decentralized controller that involves only descent directions. 
\end{theorem}

\begin{proof} 
Denote by $K^*$ a globally optimal decentralized controller. Suppose that the initialization is at a point $K_0$. If $\nabla J(K_0) = 0$, we initialize at a local minimum and can never escape. This scenario occurs on a measure-zero set since by assumption the local minima are isolated. When $\nabla J(K_0) \neq 0$, there are two cases: (1) if $\langle K^* - K, \nabla J(K) \rangle <0$, then $K^*-K$ is a descent direction, and it is possible to jump to the globally optimal solution in one step; (2) If $\langle K^* - K, \nabla J(K) \rangle \geq 0$, the global optimal solution is on the other side of the local gradient. We will prove that it is still possible turn around near a locally optimal controller, which exists in any connected components due to Lemma~\ref{lem:lqrbounded}. 

From Lemma~\ref{lem:convergelocal}, for almost any initial point $K^0$, gradient descent and a small enough step size is able to come arbitrarily close to some local minimum $K^+$. Since $K^+$ is a strict local minimum, there is some $\epsilon>0$ such that when $\|K-K^+\| \leq \epsilon$, we have $\nabla^2 J(K) \succ 0$, which means that $J$ is strongly convex in this $\epsilon$-neighborhood of $K^+$. Suppose that at $n$-th iteration, we are at some point $K^n$ in this $\epsilon$-neighborhood. It followings from strong convexity that
\begin{align*}
\langle \hspace{-0.07cm} - \hspace{-0.07cm}\nabla J(K^n), K^+\hspace{-0.1cm} -\hspace{-0.1cm} K^n\rangle \hspace{-0.07cm}=\hspace{-0.07cm} \langle \nabla J(K^+)\hspace{-0.07cm}- \hspace{-0.1cm}\nabla J(K^n), K^+ \hspace{-0.1cm}- \hspace{-0.07cm}K^n\rangle\hspace{-0.08cm} >\hspace{-0.07cm} 0
\end{align*}
which means that $K^+ - K^n$ is a descent direction at $K^n$. By continuity, there is a small $0<\delta<\epsilon$ such that $K-K^n$ is a descent direction whenever $\|K -K^+\| \leq \delta$, 
Applying Lemma~\ref{lem:find-global-direction} with this $\delta_0=\delta$ to obtain $\hat K^+$ and $\hat K^*$, and $\hat K^*$ may be selected so close to $K^*$ that gradient descent initialized at $\hat K^*$ converges to $K^*$. Assigning $K^{n+1}=\hat K^+$ and $K^{n+2}=\hat K^*$, with only descent directions we can connect $K_0$ to a neighborhood of $K^*$ where gradient descent converges. 
\end{proof}
   
\section{CONCLUSIONS}\label{sec:conclusion}
We studied the numerical behavior of local search methods when they are applied to the optimal decentralized control (ODC) problem with a disconnected feasible set. We found different behaviors between projection-based and augmented Lagrangian methods when there are multiple local minima. Moreover, we proved that a succession of jumps to the globally optimal component with descent directions is possible for almost all initializations. The existence of a path that involves many connected components provides a theoretical way to escape local minima that are created by the discontinuity of the feasible set of constrained optimal control problems. It should be noted that our existence result does not directly imply an algorithm that decides the best opportunity to take a jump to the optimal component. Using prior information about the optimal component to identify the best jump strategy is a promising direction of future research. 






\bibliographystyle{IEEEtran} 



\end{document}